\documentclass[11pt,reqno]{amsart}
\usepackage{amsmath,amssymb,extarrows}
\usepackage{url}
\usepackage{tikz,enumerate}
\usepackage{diagbox}
\usepackage{appendix}
\usepackage{epic}
\usepackage{cite}

\usepackage{hyperref}
\usepackage{array}

\usepackage{booktabs}

\allowdisplaybreaks[4]

\newtheorem{theorem}{Theorem}

\newtheorem{lemma}[theorem]{Lemma}
\newtheorem{prop}[theorem]{Proposition}
\newtheorem{exam}[theorem]{Example}
\theoremstyle{definition}

\theoremstyle{remark}
\newtheorem{rem}{Remark}
\numberwithin{equation}{section}
\numberwithin{theorem}{section}
\numberwithin{defn}{section}
\begin{document}
\title[Modular Proofs of Gosper's Identities]
 {Modular Proofs of Gosper's Identities}

\author{Liuquan Wang}
\address{School of Mathematics and Statistics, Wuhan University, Wuhan 430072, Hubei, People's Republic of China}
\email{wanglq@whu.edu.cn;mathlqwang@163.com}



%\date{
%\dedicatory{}
%\thanks{}

\begin{abstract}
We give unified modular proofs to all of Gosper's identities on the $q$-constant $\Pi_q$. We also confirm Gosper's observation that for any distinct positive integers $n_1,\cdots,n_m$ with $m\geq 3$, $\Pi_{q^{n_1}}$, $\cdots$, $\Pi_{q^{n_m}}$ satisfy a nonzero homogeneous polynomial. Our proofs provide a method to rediscover Gosper's identities. Meanwhile, several results on $\Pi_q$ found by El Bachraoui have been corrected. Furthermore, we illustrate a strategy to construct some of Gosper's identities using hauptmoduls for genus zero congruence subgroups.
\end{abstract}

\maketitle
%\tableofcontents

\section{Introduction}\label{sec-intro}
Around 2001, Gosper \cite{Gosper} gave $q$-generalizations of the trigonometric functions. He defined
\begin{align}
\sin_q(\pi z)&:=q^{(z-\frac{1}{2})^2}\frac{(q^{2z},q^{2-2z};q^2)_\infty}{(q;q^2)_\infty^2}, \\
\cos_q(\pi z)&:=q^{z^2}\frac{(q^{1+2z},q^{1-2z};q^2)_\infty}{(q;q^2)_\infty^2}.
\end{align}
Here and throughout this paper, we admit the standard $q$-series notation
\begin{align}
(a;q)_\infty:=\prod\limits_{n=0}^\infty (1-aq^n), \quad |q|<1.
\end{align}

Gosper \cite{Gosper} found many interesting formulas involving $q$-trigonometric functions such as
\begin{align}
\sin_q(2z)&=\frac{1}{2}\frac{\Pi_q}{\Pi_{q^4}}\sqrt{(\sin_{q^4}z)^2-(\sin_{q^2}z)^4}, \\
\sin_q(3z)&=\frac{1}{3}\frac{\Pi_q}{\Pi_{q^9}}\sin_{q^9}z-\left(1+\frac{1}{3}\frac{\Pi_q}{\Pi_{q^9}}\right)(\sin_{q^3}z)^3.
\end{align}
Here the $q$-constant
\begin{align}
\Pi_q:=q^{1/4}\frac{(q^2;q^2)_\infty^2}{(q;q^2)_\infty^2}
\end{align}
arises naturally and plays an important role in Gosper's $q$-trigonometry.

The  function $\Pi_q$ is closely related to Ramanujan's theta function
\begin{align}
\psi(q):=\sum_{n=0}^\infty q^{\frac{n(n+1)}{2}}.
\end{align}
From the Jacobi triple product identity, we have
\begin{align}
\Pi_q=q^{1/4}\psi(q)^2.
\end{align}

Gosper \cite{Gosper} discovered many interesting $\Pi_q$-identities, i.e., identities involving $\Pi_q$. In particular, he proved the following beautiful identity:
\begin{itemize}
\item Level 8
\begin{align}
\frac{\Pi_q^2}{\Pi_{q^2}\Pi_{q^4}}-\frac{\Pi_{q^2}^2}{\Pi_{q^4}^2}=4. \tag{L8-1} \label{G-4-1}
\end{align}
\end{itemize}
Here the term ``level'' will be explained soon. Motivated by this identity and with the help of $q$-trigonometry, Gosper searched for similar identities involving $\Pi_{q^n}$. He \cite[p.\ 94]{Gosper} made the following observation:
\begin{quote}
``In fact, it seems that for any three (or more) distinct integers $n$, $\{\Pi_{q^n}\}$ satisfy homogeneous polynomials. One can narrow the search for such relationships by noting that the factor of $q^{\frac{n}{4}}$ in $\Pi_{q^n}$ will require that in each term of the form $\Pi_{q^a}\Pi_{q^b}\Pi_{q^c}\dots$, the sum (mod 4) of $a+b+c+\cdots$ must be the same. One then proposes a polynomial with undetermined coefficients, Taylor expands it, and attempts to solve the linear system resulting from equating powers of $q$. ...... Here are some typical identities so discovered (but not proved!):''
\end{quote}
Gosper's identities present polynomial relations among the $\Pi_{q^n}$'s. For our convenience, for an identity involving $\Pi_{q^{n_1}}, \Pi_{q^{n_2}}, \cdots,\Pi_{q^{n_m}}$ with $n_1<n_2<\cdots<n_m$, we call $N=2\mathrm{lcm}\{n_1,n_2,\cdots,n_m\}$, i.e., twice of the least common multiple of $n_1,n_2,\cdots,n_m$ as the \textit{level} of this identity. We will see later that an identity of level $N$ is essentially an equality of modular forms on $\Gamma_0(N)$. Thus the level $N$ is related to the level of modular forms involved in the identity.

Here are the full list of identities found by Gosper \cite[pp.\ 95,103,104]{Gosper} by computer without proof. We reorder the identities according to their levels.
\begin{itemize}
\item Level 12 \footnote{We use the tag L$N$-$n$ to denote the $n$-th identity of level $N$ in our list.}
\begin{align}
\Pi_{q^2}^2+2\Pi_{q^2}\Pi_{q^6}&=\Pi_q\Pi_{q^3}+3\Pi_{q^6}^2, \tag{L12-1} \label{G-6-1} \\
\frac{\Pi_{q^2}\Pi_{q^3}^2}{\Pi_{q^6}\Pi_{q}^2}&=\frac{\Pi_{q^2}-\Pi_{q^6}}{\Pi_{q^2}+3\Pi_{q^6}}, \tag{L12-2} \label{G-6-2} \\
\sqrt{\Pi_{q^2}\Pi_{q^6}}\left(\Pi_q^2-3\Pi_{q^3}^2 \right)&=\sqrt{\Pi_q\Pi_{q^3}}\left(\Pi_{q^2}^2+3\Pi_{q^6}^2 \right), \tag{L12-3} \label{G-6-3}\\
\Pi_{q^2}\Pi_{q^3}^4&=\Pi_{q^6}\left(\Pi_{q^2}-\Pi_{q^6} \right)^3\left(\Pi_{q^2}+3\Pi_{q^6}\right), \tag{L12-4} \label{G-6-4}\\
\Pi_{q^6}\Pi_q^4&=\Pi_{q^2}\left(\Pi_{q^2}-\Pi_{q^6} \right)\left(\Pi_{q^2}+3\Pi_{q^6} \right)^3, \tag{L12-5} \label{G-6-5}\\
\Pi_q\Pi_{q^3}\left(\Pi_{q}^2\pm 4\Pi_{q^2}^2 \right)^2&=\Pi_{q^2}^2\left(\Pi_q\mp \Pi_{q^3} \right)\left(\Pi_q\pm 3\Pi_{q^3} \right)^3, \tag{L12-6} \label{G-6-6}\\
\Pi_q\Pi_{q^3}\left(\Pi_{q^3}^2\pm 4\Pi_{q^6}^2 \right)^2&=\Pi_{q^6}^2\left(\Pi_q\mp \Pi_{q^3} \right)^3\left(\Pi_q\pm 3\Pi_{q^3} \right), \tag{L12-7} \label{G-6-7} \\
\Pi_{q^2}^2\left(\Pi_q^4+18\Pi_q^2\Pi_{q^3}^2-27\Pi_{q^3}^4 \right)&=\Pi_q\Pi_{q^3}\left(\Pi_q^4+16\Pi_{q^2}^4 \right), \tag{L12-8} \label{G-6-8} \\
\Pi_{q^6}^2\left(\Pi_q^4-6\Pi_q^2\Pi_{q^3}^2-3\Pi_{q^3}^4 \right)&=\Pi_q\Pi_{q^3}\left(\Pi_{q^3}^4+16\Pi_{q^6}^4 \right). \tag{L12-9} \label{G-6-9}
\end{align}
\item Level 18
\begin{align}
\Pi_{q^3}^2+3\Pi_q\Pi_{q^9}=\sqrt{\Pi_q\Pi_{q^9}}\left(\Pi_q+3\Pi_{q^9} \right).  \tag{L18-1} \label{G-9-1}
\end{align}
\item Level 20
\begin{align}
\Pi_{q^2}\Pi_{q^5}^4\left(16\Pi_{q^{10}}^4-\Pi_{q^5}^4 \right)&=\Pi_{q^{10}}^3\left(5\Pi_{q^{10}}-\Pi_{q^2} \right) \left(\Pi_{q^2}-\Pi_{q^{10}} \right)^5, \tag{L20-1} \label{G-10-1} \\
\Pi_{q^{10}}\Pi_q^4\left(16\Pi_{q^2}^4-\Pi_q^4 \right)&=\Pi_{q^2}^3\left(5\Pi_{q^{10}}-\Pi_{q^2} \right)^5\left(\Pi_{q^2}-\Pi_{q^{10}} \right),  \tag{L20-2}\label{G-10-2} \\
\Pi_q\Pi_{q^5}\left(16\Pi_{q^2}^4-\Pi_q^4 \right)^2&=\Pi_{q^2}^4\left(5\Pi_{q^5}-\Pi_q \right)^5\left(\Pi_{q^5}-\Pi_q \right), \tag{L20-3} \label{G-10-3}\\
\Pi_q\Pi_{q^5}\left(16\Pi_{q^{10}}^4-\Pi_{q^5}^4 \right)^2&=\Pi_{q^{10}}^4\left(5\Pi_{q^5}-\Pi_q \right)\left(\Pi_{q^5}-\Pi_q \right)^5,  \tag{L20-4} \label{G-10-4}\\
\Pi_{q^2}\Pi_{q^{10}}\left(\Pi_{q^5}-\Pi_q \right)\left(5\Pi_{q^5}-\Pi_q \right)&=\left(\Pi_q\Pi_{q^{10}}-\Pi_{q^2}\Pi_{q^5} \right)^2. \tag{L20-5} \label{G-10-5}
\end{align}
\end{itemize}

Gosper's identities have been investigated by a number of researchers. In 2019, He \cite{He2019} gave proofs of all the level 12 and level 20 identities of Gosper stated above except for \eqref{G-6-1}. His proofs rely on using some modular equations of degrees 3 and 5. He and Zhai \cite{He-Zhai} confirmed \eqref{G-9-1}.

Earlier than He's work \cite{He2019}, El Bachraoui \cite{Bachraoui} proved that the square of both sides of \eqref{G-6-2} are equal. He did not realize that this fact actually proves \eqref{G-6-2} with few more arguments (see Lemma \ref{lem-square} and Remark \ref{rem-square} in this paper). El Bachraoui also proved the equivalence of \eqref{G-6-4} and \eqref{G-6-5}. Furthermore, El Bachraoui reproved \eqref{G-4-1} (see \cite[Theorem 4.1(a)]{Bachraoui}) and stated some new identities \cite[Theorems 4.1--4.3]{Bachraoui}. However, some of his identities are incorrect or contain typos (see Remark \ref{rem-Bachraoui-incorrect}). Below we state the  correct identities. We also include a level 16 identity
in the work of Abo Touk et al. \cite[Theorem 5(a)]{AboTouk}.
\begin{itemize}
\item Level 12
\begin{align}
(\Pi_{q}^2-\Pi_{q^3}^2)^4\Pi_{q^3}\Pi_{q^2}^3\Pi_{q^6}&=(\Pi_{q^2}^2-\Pi_{q^6}^2)^2\Pi_{q}^3(4\Pi_{q^6}^2+\Pi_q\Pi_{q^3})^3, \tag{L12-10} \label{L6-10} \\
\Pi_{q^3}\Pi_{q^6}^4(\Pi_q^2-\Pi_{q^3}^2)^4&=(\Pi_{q^2}^2-\Pi_{q^6}^2)^2(\Pi_{q^3}^3+4\Pi_q\Pi_{q^6}^2)^3, \tag{L12-11} \label{L6-11} \\
(\Pi_q^2-\Pi_{q^3}^2)(\Pi_{q^2}-\Pi_{q^6})^3\Pi_{q^6}&=\Pi_{q^3}^2(\Pi_{q^2}^2-\Pi_{q^6}^2)^2, \tag{L12-12} \label{L6-12} \\
\Pi_{q^2}\Pi_{q^3}^3&=\Pi_q\Pi_{q^6}\left(\Pi_{q^2}-\Pi_{q^6} \right)^2, \tag{L12-13} \label{L6-13} \\
\Pi_q^3\Pi_{q^6}&=\Pi_{q^2}\Pi_{q^3}\left(\Pi_{q^2}+3\Pi_{q^6} \right)^2. \tag{L12-14} \label{L6-14}
\end{align}
\item Level 16
\begin{align}
\Pi_q^2\Pi_{q^8}=\Pi_{q^2}\Big(\Pi_{q^4}+2\Pi_{q^8}\Big)^2. \tag{L16-1} \label{L8-1}
\end{align}
\item Level 18
\begin{align}
\sqrt{\Pi_{q^9}}\left(\sqrt{\Pi_{q}}-\sqrt{\Pi_{q^9}} \right)^3&=\Pi_{q^3}^2-\Pi_{q^9}^2, \tag{L18-2}\label{L9-2} \\
\Pi_{q^9}^2(\Pi_q^2-\Pi_{q^3}^2)^3&=(\Pi_{q^3}^2-\Pi_{q^9}^2)(\Pi_{q^3}^2+3\Pi_q\Pi_{q^9})^3, \tag{L18-3} \label{L9-3} \\
\Pi_{q^9}\left(\Pi_q^2-\Pi_{q^3}^2 \right)^6&=\Pi_q^3\left(\Pi_{q^3}^2-\Pi_{q^9}^2 \right)^2\left(\Pi_q+3\Pi_{q^9}\right)^6. \tag{L18-4} \label{L9-4}
\end{align}
\end{itemize}
\begin{rem}\label{rem-Bachraoui-incorrect}
(1) The identity in part (a) of \cite[Theorem 4.1]{Bachraoui} is the same with \eqref{G-4-1}. The identities in parts (b,c,d) of
\cite[Theorem 4.1]{Bachraoui} are incorrect. The mistakes are probably caused by incorrect manipulations. The identities \eqref{L6-10}, \eqref{L6-11} and \eqref{L6-12} are the correct identities corresponding to (b),(c),(d), respectively. We recover these identities by checking the proofs given by El Bachraoui in \cite{Bachraoui}.

(2) Identity \eqref{L6-13} is part (a) of
 \cite[Theorem 4.2]{Bachraoui}. Note that part (b) of \cite[Theorem 4.2]{Bachraoui} is incorrect. In fact, if we follow the proof of (b) in \cite{Bachraoui}, we get nothing but a trivial identity.

(3) Identities \eqref{L6-14}, \eqref{L9-2}, \eqref{L9-3} and \eqref{L9-4} are parts (a),(b),(c) and (d) of \cite[Theorem 4.3]{Bachraoui}, respectively. Note that the original identity in part (c) stated in  \cite[Theorem 4.3]{Bachraoui} is incorrect. The correct one should be \eqref{L9-3}. The identity \eqref{L6-14} also appeared in \cite[Theorem 5(b)]{AboTouk}.

(4) El Bachraoui \cite[Theorem 2.1]{Bachraoui} claimed that he found the following identities which involve only two of the $\Pi_{q^n}$:
\begin{align}
&\Pi_{q^9}^2\left(\Pi_{q}^2-\Pi_{q^9}^2 -\sqrt{\Pi_{q^9}}(\sqrt{\Pi_q}-\sqrt{\Pi_{q^9}})^3 \right)^3 \nonumber \\
&= \left(\Pi_{q^9}^2+\sqrt{\Pi_{q^9}}(\sqrt{\Pi_q}-\sqrt{\Pi_{q^9}})^3+3\Pi_q\Pi_{q^9} \right)^3  \sqrt{\Pi_{q^9}}(\sqrt{\Pi_q}-\sqrt{\Pi_{q^9}})^3,  \label{False-L9-5} \\
&\Pi_{q^9}\left(\Pi_q^2-\Pi_{q^9}^2-\sqrt{\Pi_{q^9}}\Big(\sqrt{\Pi_q}-\sqrt{\Pi_{q^9}}\Big)^3\right)^6\nonumber \\
&=\Pi_q^3\Pi_{q^9}\Big(\sqrt{\Pi_q}-\sqrt{\Pi_{q^9}}\Big)^6\Big(\Pi_q+3\Pi_{q^9}\Big)^6.   \label{False-L9-6}
\end{align}
The identities \eqref{False-L9-5} and \eqref{False-L9-6} correspond to parts (a) and (b) of \cite[Theorem 2.1]{Bachraoui}, respectively. Part (a) is deduced from \eqref{L9-2} and \eqref{L9-3} by eliminating $\Pi_{q^3}$. Part (b) is deduced from \eqref{L9-2} and \eqref{L9-4} in a similar way. Note that the original identity for part (a) is incorrect and there is  also a typo for part (b). We have corrected both of them here. It should be noted that the identities \eqref{False-L9-5} and \eqref{False-L9-6} are in fact trivial identities which have nothing to do with $\Pi_q$ or $\Pi_{q^9}$. In fact, if we denote $\sqrt{\Pi_q}$ and $\sqrt{\Pi_{q^9}}$ by $a$ and $b$, respectively, then both sides of \eqref{False-L9-5} (resp.\ \eqref{False-L9-6}) are equal to $a^3b^4(a-b)^3(a^2+3b^2)^3$ (resp.\ $a^6b^2(a-b)^6(a^2+3b^2)^6$).

In fact, given any positive integers $n_1<n_2$, it is impossible to find nontrivial polynomial relations between $\sqrt{\Pi_{q^{n_1}}}$ and $\sqrt{\Pi_{q^{n_2}}}$. This can be proved in a way similar to the proof of the Corollary given in \cite[p.\ 15]{Zagier}.
\end{rem}

Gosper \cite{Gosper} also stated a set of Lambert series identities related to $\Pi_q$. Again, here we present his identities according to the level of the parts involving $\Pi_{q}$.
\begin{itemize}
\item Level 2 \footnote{We use the tag La$N$-$n$ to denote the $n$-th Lambert series identity of level $N$ in our list.}
\begin{align}
\Pi_q^4=6\sum_{n\geq 1}\frac{q^{4n-2}}{(1-q^{2n-1})^4}+\sum_{n\geq 1}\frac{q^{2n-1}}{(1-q^{2n-1})^2}=\sum_{n\geq 1}\frac{n^3q^n}{1-q^{2n}}. \tag{La2-1} \label{LL1-1}
\end{align}
\item Level 4
\begin{align}
\sum_{n\geq 1}\frac{q^n}{(1-q^n)^2}-2\sum_{n\geq 1}\frac{q^{2n}}{(1-q^{2n})^2}&=\frac{1}{24}\left(\frac{\Pi_q^4}{\Pi_{q^2}^2}-1 \right)+\frac{2}{3}\Pi_{q^2}^2, \tag{La4-1} \label{LL2-1} \\
\sum_{n\geq 1} \frac{q^n}{(1-q^n)^2}-4\sum_{n\geq 1}\frac{q^{4n}}{(1-q^{4n})^2}&=\frac{1}{8}\left(\frac{\Pi_q^4}{\Pi_{q^2}^2}-1 \right), \tag{La4-2} \label{LL2-2} \\
\sum_{n\geq 1} \frac{q^{2n-1}}{(1-q^{2n-1})^2}-2\sum_{n\geq 1}\frac{q^{4n-2}}{(1-q^{4n-2})^2}&=\Pi_{q^2}^2=\sum_{n\geq 1}\frac{(2n-1)q^{2n-1}}{1-q^{4n-2}}. \tag{La4-3} \label{LL2-3}
\end{align}
\item Level 6
\begin{align}
\sum_{n\geq 1}\frac{q^n}{(1-q^n)^2}-3\sum_{n\geq 1}\frac{q^{3n}}{(1-q^{3n})^2}&=\frac{\left(\Pi_q^2+3\Pi_{q^3}^2 \right)^2}{12\Pi_q\Pi_{q^3}}-\frac{1}{12}, \tag{La6-1} \label{LL3-1} \\
\sum_{n\geq 1}\frac{q^{2n-1}}{(1-q^{2n-1})^2}-3\sum_{n\geq 1}\frac{q^{6n-3}}{(1-q^{6n-3})^2}&=\Pi_q\Pi_{q^3}. \tag{La6-2} \label{LL3-2}
\end{align}
\item Level 8
\begin{align}
\sum_{n\geq 1}\frac{q^{2n-1}}{(1-q^{2n-1})^2}-4\sum_{n\geq 1}\frac{q^{8n-4}}{(1-q^{8n-4})^2}=\frac{1}{8}\left(\frac{\Pi_q^4}{\Pi_{q^2}^2}-\frac{\Pi_{q^2}^4}{\Pi_{q^4}^2} \right)=\Pi_{q^2}^2+2\Pi_{q^4}^2. \tag{La8-1} \label{LL4-1}
\end{align}
\item Level 10
\begin{align}
&6\Big(\sum_{n\geq 1}\frac{q^n}{(1-q^n)^2}-5\sum_{n\geq 1}\frac{q^{5n}}{(1-q^{5n})^2} \Big)+1 \nonumber \\
&=\left(\frac{\Pi_q}{\Pi_{q^5}}+2+5\frac{\Pi_{q^5}}{\Pi_q} \right)\Big(\sum_{n\geq 1}\frac{q^{2n-1}}{(1-q^{2n-1})^2}-5\sum_{n\geq 1}\frac{q^{10n-5}}{(1-q^{10n-5})^2} \Big), \tag{La10-1} \label{LL5-1} \\
&\frac{\sum_{n\geq 1}\frac{q^{2n-1}}{(1-q^{2n-1})^2}-5\sum_{n\geq 1}\frac{q^{10n-5}}{(1-q^{10n-5})^2}}{\Pi_{q^5}^2}=\sqrt{\frac{\Pi_q^3}{\Pi_{q^5}^3}-2\frac{\Pi_q^2}{\Pi_{q^5}^2}+5\frac{\Pi_q}{\Pi_{q^5}}}. \tag{La10-2} \label{LL5-2}
\end{align}
\item Level 12
\begin{align}
\sum_{n\geq 1}\frac{q^{2n-1}}{(1-q^{2n-1})^2}-6\sum_{n\geq 1}\frac{q^{12n-6}}{(1-q^{12n-6})^2}=\Pi_{q^2}^2+2\Pi_{q^2}\Pi_{q^6}=\Pi_q\Pi_{q^3}+3\Pi_{q^6}^2. \tag{La12-1} \label{LL6-1}
\end{align}
\item Level 18
\begin{align}
&\sum_{n\geq 1}\frac{q^{2n}}{(1-q^{2n})^2}-9\sum_{n\geq 1}\frac{q^{18n}}{(1-q^{18n})^2}=\frac{\Pi_{q^3}^3}{\Pi_q}+\frac{1}{3}\Big(\frac{\Pi_{q^3}^3}{\Pi_{q^9}}-1 \Big), \tag{La18-1} \label{LL9-1} \\
&\sum_{n\geq 1}\frac{q^{2n-1}}{(1-q^{2n-1})^2}-9\sum_{n\geq 1}\frac{q^{18n-9}}{(1-q^{18n-9})^2} \nonumber \\
&=\left(\Pi_q\Pi_{q^9}+3\Pi_{q^9}^2 \right)\sqrt{\left(\frac{\Pi_q}{\Pi_{q^9}} \right)^{\frac{3}{2}}-3\frac{\Pi_q}{\Pi_{q^9}}+3\left(\frac{\Pi_q}{\Pi_{q^9}} \right)^{\frac{1}{2}}}, \tag{La18-2} \label{LL9-2}\\
&3\Big(\sum_{n\geq 1}\frac{q^n}{(1-q^n)^2}-9\sum_{n\geq 1}\frac{q^{9n}}{(1-q^{9n})^2} \Big)+1  \nonumber \\
&=\Big(\sqrt{\frac{\Pi_q}{\Pi_{q^9}}}+3\sqrt{\frac{\Pi_{q^9}}{\Pi_q}} \Big)\Big(\sum_{n\geq 1}\frac{q^{2n-1}}{(1-q^{2n-1})^2}-9\sum_{n\geq 1}\frac{q^{18n-9}}{(1-q^{18n-9})^2} \Big).  \tag{La18-3} \label{LL9-3}
%&=\left(\Pi_q+3\Pi_{q^9} \right)^3\left(\Pi_q\Pi_{q^9}+\Pi_{q^3}^2 \right)^2+\Pi_q\Pi_{q^9}\frac{\left(\Pi_q-3\Pi_{q^9} \right)^2}{4\Pi_q\Pi_{q^3}\Pi_{q^9}}.
\end{align}
\item Level 20
\begin{align}
&\frac{\sum_{n\geq 1}\frac{q^{2n-1}}{(1-q^{2n-1})^2}-5\sum_{n\geq 1}\frac{q^{10n-5}}{(1-q^{10n-5})^2}}{\Pi_{q^5}^2}=\frac{\frac{\Pi_{q^5}^2}{\Pi_{q^{10}}^2}+16\frac{\Pi_{q^{10}}^2}{\Pi_{q^5}^2} }{\frac{\Pi_q}{\Pi_{q^5}}-4-\frac{\Pi_{q^5}}{\Pi_q}}. \tag{La20-1} \label{LL20-1}
\end{align}
\end{itemize}
\begin{rem}\label{rem-Gosper-wrong}
(1) Gosper stated \eqref{LL5-2} and \eqref{LL20-1} in a single equation. Here we split them according to different levels.

(2) Gosper \cite[p.\ 104]{Gosper} also stated another expression for the series in \eqref{LL9-3}. But his expression obviously contains typos and is incorrect. We do not know what he intended to write here and hence do not include that wrong formula. Nevertheless, we find a new expression for the series in \eqref{LL9-3} and an associated $\Pi_q$-identity:
\begin{itemize}
\item Level 18
\begin{align}
3\Big(\sum_{n\geq 1}\frac{q^n}{(1-q^n)^2}-9\sum_{n\geq 1}\frac{q^{9n}}{(1-q^{9n})^2} \Big)+1  &=3\Pi_q\Pi_{q^3}+9\Pi_{q^3}\Pi_{q^9}+3\frac{\Pi_{q^3}^3}{\Pi_{q}}+\frac{\Pi_{q^3}^3}{\Pi_{q^9}}, \tag{La18-4} \label{La18-4} \\
6\frac{\Pi_q}{\Pi_{q^3}}+3\frac{\Pi_{q^3}}{\Pi_{q}}-\frac{\Pi_q^3}{\Pi_{q^3}^3}&=27\frac{\Pi_{q^9}^3}{\Pi_{q^3}^3}-18\frac{\Pi_{q^9}}{\Pi_{q^3}}-\frac{\Pi_{q^3}}{\Pi_{q^9}}. \tag{L18-5} \label{L18-5}
\end{align}
\end{itemize}
\end{rem}

In \cite{Gosper}, Gosper did not provide any proofs of the above Lambert series identities. El Bachraoui \cite[Theorem 1]{Bachraoui-PAMS} proved the identities \eqref{LL1-1}, \eqref{LL2-1}, \eqref{LL2-2},  \eqref{LL3-1}, \eqref{LL3-2} and \eqref{LL9-1}. His proofs are based on Gosper's $q$-trigonometry, which are perhaps close to the idea behind Gosper's discovery of these identities.
Later He \cite{HeAAM} gave a proof of \eqref{LL9-2} based on the theory of modular forms.
%Gosper also stated without proof many identities relating $\Pi_q$ with Lambert series. For example, he \cite{Gosper} state that
%\begin{align}
%&6\left(\sum_{n\geq 1}\frac{q^n}{(1-q^n)^2}-5\sum_{n\geq 1}\frac{q^{5n}}{(1-q^{5n})^2} \right)+1 \\
%&=\left(\frac{\Pi_q}{\Pi_{q^5}}+2+5\frac{\Pi_{q^5}}{\Pi_q} \right)\left(\sum_{n\geq 1}\frac{q^{2n-1}}{(1-q^{2n-1})^2}-5\sum_{n\geq 1}\frac{q^{10n-5}}{(1-q^{10n-5})^2} \right), \label{LL5-1} \\
%&\frac{\sum_{n\geq 1}\frac{q^{2n-1}}{(1-q^{2n-1})^2}-5\frac{q^{10n-5}}{(1-q^{10n-5})^2}}{\Pi_{q^5}^2}=\frac{\frac{\Pi_{q^5}^2}{\Pi_{q^{10}}^2}+16\frac{\Pi_{q^{10}}^2}{\Pi_{q^5}^2} }{\frac{\Pi_q}{\Pi_{q^5}}-4-\frac{\Pi_{q^5}}{\Pi_q}}=\sqrt{\frac{\Pi_q^3}{\Pi_{q^5}^3}-2\frac{\Pi_q^2}{\Pi_{q^5}^2}+5\frac{\Pi_q}{\Pi_{q^5}}}. \label{Ll5-2-intro}
%\end{align}
%Some of his Lambert series identities were proved by El Bachrouai \cite{Bachraoui-PAMS} and He \cite{HeAAM}.

Motivated by Gosper's observation and the above works, there are two goals of this paper. First, we will confirm Gosper's observation, which is actually a standard fact in the theory of modular forms.
\begin{theorem}\label{thm-main}
Gosper's observations is true. That is, for any distinct positive integers $n_1<\cdots<n_m$ with $m\geq 3$, there is a nonzero homogeneous polynomial $f(x_1,\cdots,x_m)\in \mathbb{C}[x_1,\cdots,x_m]$ such that
$$f(\Pi_{q^{n_1}},\cdots,\Pi_{q^{n_m}})=0.$$
%Here $k_0$ is the least positive integer $k$ such that
%\begin{align}\label{k0-defn}
%\binom{2k+m-1}{m-1}>\frac{1}{2}kN\prod\limits_{p|N,~p>2}\left(1+\frac{1}{p}\right)+1
%\end{align}
%with $N=2\mathrm{lcm}\{n_1,\cdots,n_m\}$.
\end{theorem}

Second, we will give unified modular proofs for all of Gosper's $\Pi_q$-identities. As a byproduct of our method, besides \eqref{La18-4} and \eqref{L18-5}, we also find the following new identity:
\begin{itemize}
\item Level 16
\begin{align}
\Pi_q^4\Pi_{q^4}\Pi_{q^8}\Big(\Pi_{q^4}^2+4\Pi_{q^8}^2\Big)=\Pi_{q^2}^4\Big(\Pi_{q^4}+2\Pi_{q^8} \Big)^4. \tag{L16-2} \label{L8-2}
\end{align}
\end{itemize}

\begin{theorem}\label{thm-id}
All the identities \eqref{G-4-1}--\eqref{L9-4}, \eqref{L8-2} and \eqref{L18-5} are true.
\end{theorem}

\begin{theorem}\label{thm-Lambert}
All the identities \eqref{LL1-1}--\eqref{LL20-1} and \eqref{La18-4} are true.
\end{theorem}

The paper is organized as follows. In Section \ref{sec-proof} we first collect some useful facts from the theory of modular forms, and then we give proofs for Theorems \ref{thm-main} and \ref{thm-id}. In Section \ref{sec-Lambert} we present proof for Theorem \ref{thm-Lambert}. We emphasize here that our proofs give not only a way of verification  but also a method for rediscovering Gospers' identities. The strategy for discovering Gosper's identities will be explained during the proofs. Finally, in Section \ref{sec-hauptmodul} we illustrate a different method using hauptmoduls on congruence subgroups of genus zero. This method can explain some of Gosper's identities very well.

\section{Proofs of Gosper's $\Pi_q$-identities}\label{sec-proof}
We first recall some basic results about eta products. The Dedekind eta function is defined as
\begin{align}
\eta(z):=q^{1/24}\prod\limits_{n=1}^\infty (1-q^n), \quad q=e^{2\pi iz}, \quad \mathrm{Im} z>0.
\end{align}
We call
\begin{align}\label{eta-prod}
f(z)=\prod\limits_{\delta|N}\eta(\delta z)^{r_\delta}
\end{align}
as an eta product or eta quotient.

Let $\mathbb{H}$ be the upper half plane, i.e., $\mathbb{H}=\{z: \mathrm{Im} z>0\}$. Let $\mathbb{H}^*=\mathbb{H}\cup \mathbb{Q}\cup \{\infty\}$. For any meromorphic function $f:\mathbb{H}^{*}\rightarrow \mathbb{C}$ and $p\in \mathbb{H}^*$, we use $\mathrm{ord}(f,p)$ to denote the order of $f$ at $p$.

The full modular group is given by
\begin{align}
\mathrm{SL}_2(\mathbb{Z}):=\left\{\begin{pmatrix}
a & b \\ c & d
\end{pmatrix}: ad-bc=1, a,b,c,d\in \mathbb{Z}\right\}.
\end{align}
The following lemmas give sufficient conditions for an eta product to be a modular form on the congruence subgroup
\begin{align}
\Gamma_0(N):=\left\{\begin{pmatrix}
a & b \\ c & d
\end{pmatrix}\in \mathrm{SL}_2(\mathbb{Z}): c\equiv 0 \!\!\!\pmod{N}\right\}.
\end{align}
\begin{lemma}\label{lem-modular}
(Cf.\ \cite[Theorem 1.64]{Ono-web}.) If $f(z)$ is an eta product given in \eqref{eta-prod} with $k=\frac{1}{2}\sum_{\delta|N} r_\delta \in \mathbb{Z}$ with the additional properties that
\begin{align}
\sum_{\delta |N}\delta r_\delta\equiv 0 \pmod{24}
\end{align}
and
\begin{align}
\sum_{\delta| N}\frac{N}{\delta}r_\delta\equiv 0 \pmod{24},
\end{align}
then $f(z)$ satisfies
\begin{align}
f\left(\frac{az+b}{cz+d} \right)=\chi(d)(cz+d)^kf(z)
\end{align}
for every $\begin{pmatrix}
a & b \\ c &d
\end{pmatrix}\in \Gamma_0(N)$. Here the character $\chi$ is defined by $\chi(d):=\left(\frac{(-1)^ks}{d}\right)$ where $s:=\prod_{\delta|N}\delta^{r_\delta}$.
\end{lemma}

\begin{lemma}\label{lem-ord}
(Cf.\ \cite[Theorem 1.65]{Ono-web}.) Let $r,s$ and $N$ be positive integers with $s|N$ and $\gcd(r,s)=1$. If $f(z)$ is an eta product satisfying the conditions of Lemma \ref{lem-modular} for $N$, then
\begin{align}
\mathrm{ord}(f,\frac{r}{s})=\frac{N}{24}\sum_{\delta |N}\frac{\gcd(s,\delta)^2r_\delta}{\gcd(s,\frac{N}{s})s\delta}.
\end{align}
\end{lemma}
Let $M_k(\Gamma_0(N),\chi)$ denote the space of modular forms on $\Gamma_0(N)$ with Nebentypus $\chi$. In particular, when $\chi$ is trivial, we also write it as $M_k(\Gamma_0(N))$. If $f(z)$ is an eta-product satisfying the conditions of Lemma \ref{lem-modular} and its orders at cusps of $\Gamma_0(N)$ are all nonnegative, then $f(z)$ belongs to  $M_k(\Gamma_0(N),\chi)$.

As a consequence of Lemmas \ref{lem-modular} and \ref{lem-ord}, we have the following proposition, which is the starting point of our proofs of the theorems.
\begin{prop}\label{prop}
Let $f(z)=\Pi_{q^{n_1}}^{k_1}\cdots\Pi_{q^{n_m}}^{k_m}$ with $1\leq n_1<\cdots <n_m$ and $k_1,\cdots,k_m\in \frac{1}{2}\mathbb{Z}$. Let $N:=2\mathrm{lcm}\{n_1,\cdots,n_m\}$.  Suppose that
\begin{align}
k:=k_1+\cdots+k_m & \in \mathbb{Z}, \label{cond-1} \\
k_1n_1+\cdots+k_mn_m & \equiv 0 \pmod{4}. \label{cond-2}
\end{align}
Then $f(z)$ satisfies
\begin{align}
f\left(\frac{az+b}{cz+d} \right)=\chi(d)(cz+d)^kf(z)
\end{align}
for every $\begin{pmatrix}
a & b \\ c &d
\end{pmatrix}\in \Gamma_0(N)$ where $\chi(d):=\left(\frac{(-1)^k}{d}\right)$. Furthermore, let $r,s$ be positive integers with $s|N$ and $\gcd(r,s)=1$, we have
\begin{align}
\mathrm{ord}(f(z),\frac{r}{s})=\frac{N}{24s\gcd(s,\frac{N}{s})}\sum_{i=1}^m\frac{2k_i}{n_i}\left(\gcd(s,2n_i)^2-\gcd(s,n_i)^2  \right).
\end{align}
If in addition we have for any $s|N$ that
\begin{align}\label{cond-3}
\sum_{i=1}^m \frac{k_i}{n_i}\left(\gcd(s,2n_i)^2-\gcd(s,n_i)^2 \right)\geq 0,
\end{align}
then $f(z)\in M_{k}(\Gamma_0(N),\chi)$.
\end{prop}
\begin{proof}
We first note that
\begin{align}
\Pi_q=\frac{\eta^4(2z)}{\eta^2(z)}.
\end{align}
We have
\begin{align}
f(z)=\prod\limits_{i=1}^m \frac{\eta^{4k_i}(2n_iz)}{\eta^{2k_i}(n_iz)}.
\end{align}
The proposition then follows from Lemmas \ref{lem-modular} and \ref{lem-ord}.
\end{proof}

The following result, known as Sturm's criterion, will also be needed to verify the equalities between modular forms.
\begin{lemma}\label{lem-Sturm}
(Cf.\ \cite{Sturm}) Let $\Gamma\subseteq \mathrm{SL}_2(\mathbb{Z})$ be a congruence subgroup and let $f\in M_k(\Gamma)$. If
\begin{align}
\mathrm{ord}(f,\infty)>\frac{k}{12}[\mathrm{SL}_2(\mathbb{Z}):\Gamma],
\end{align}
then $f$ is identically zero.
\end{lemma}
For our cases, we always choose $\Gamma=\Gamma_0(N)$ and hence the index is \cite[p.\ 14]{Diamond-Shurman}
\begin{align}\label{index}
[\mathrm{SL}_2(\mathbb{Z}):\Gamma_0(N)]=N\prod\limits_{p|N}\Big(1+\frac{1}{p}\Big).
\end{align}

We are now ready to prove Theorem \ref{thm-main}.
\begin{proof}[Proof of Theorem \ref{thm-main}]
Let $n=\mathrm{lcm}\{n_1,n_2,\cdots,n_m\}$ and $k$ be a positive integer. We define
\begin{align}
S:=\left\{\Pi_{q^{4n_1}}^{k_1}\cdots \Pi_{q^{4n_m}}^{k_m}: k_1+k_2+\cdots +k_m=2k,k_1,k_2,\cdots,k_m\geq 0\right\}.
\end{align}
Then the conditions \eqref{cond-1}, \eqref{cond-2} and \eqref{cond-3} are satisfied. Hence by Proposition \ref{prop}, $S$ is a subset of $M_{2k}(\Gamma_0(8n))$. Now recall that (see \cite[p.\ 12, Proposition 3]{Zagier})
\begin{align}
\dim M_{2k}(\Gamma_0(8n))\leq \frac{k}{6} [\mathrm{SL}_2(\mathbb{Z}):\Gamma_0(8n)]+1=2kn\prod\limits_{p|n,~p>2}\left(1+\frac{1}{p}\right)+1.
\end{align}
Note that
\begin{align}
|S|=\binom{2k+m-1}{m-1}.
\end{align}
Since $m\geq 3$, as functions of $k$, $|S|$ grows faster than $\dim M_{2k}(\Gamma_0(8n))$. Hence there exists some $k$ such that
\begin{align}\label{k-ineq}
\binom{2k+m-1}{m-1}>\dim M_{2k}(\Gamma_0(8n)).
\end{align}
Let $k_0$ be the smallest one among such $k$'s. Then for $k=k_0$, the elements in $S$ must be linearly dependent. Therefore, the functions $\Pi_{q^{4n_1}}, \cdots,\Pi_{q^{4n_m}}$ satisfy a nonzero homogeneous polynomial of degree $2k_0$. Replacing $q^4$ by $q$, we get the desired assertion.
\end{proof}
Note that $k_0$ may not be the smallest degree of nonzero homogeneous polynomial satisfied by $\Pi_{q^{n_1}}, \cdots, \Pi_{q^{n_m}}$. This can be seen from the following example.
\begin{exam}
For the index set $\{n_1,n_2,n_3,n_4\}=\{1,2,5,10\}$, we have $m=4$ and $n=10$. We find that the smallest $k$ satisfying \eqref{k-ineq} is $k_0=3$. Indeed, we have  $\dim M_6(\Gamma_0(80))=66$, which is less than $|S|=\binom{9}{3}=84$. Thus we deduce that $\Pi_{q}$, $\Pi_{q^2}$, $\Pi_{q^5}$, $\Pi_{q^{10}}$ must satisfy a homogeneous polynomial of degree $6$. It turns out that they satisfy a simpler homogenous polynomial of degree $4$ as in \eqref{G-10-5}.
\end{exam}

Before proving Theorem \ref{thm-id}, we need the following simple fact.
\begin{lemma}\label{lem-square}
Let
 $$f(q)=\sum_{n=n_1}^\infty a_nq^n \quad \text{and} \quad  g(q)=\sum_{n=n_2}^\infty b_nq^n$$
where $a_n,b_n\in \mathbb{R}$, $n_1,n_2\geq 0$, $a_{n_1}\neq 0$, $b_{n_2}\neq 0$. Suppose both $f(q)$ and $g(q)$ are holomorphic in the open unit disk $|q|<1$ and $f(q)^\ell=g(q)^\ell$ where $\ell$ is a positive integer. Then $n_1=n_2$ and there exists some $0\leq j <\ell$ such that $a_{n_1}=e^{2\pi ij/\ell}b_{n_2}$ and $f(q)=e^{2\pi ij/\ell}g(q)$.
\end{lemma}
\begin{proof}
Since $\prod\limits_{j=0}^{\ell-1}(f(q)-e^{2\pi ij/\ell}g(q))=0$, we see that one of the sets $\{q: f(q)=e^{2\pi ij/\ell}g(q), |q|\leq \frac{1}{2}\}$ contains infinitely many points. By the theory of complex functions, this means that for some $0<j\leq \ell$, $f(q)=e^{2\pi ij/\ell}g(q)$ for all $|q|< 1$.
\end{proof}

\begin{proof}[Proof of Theorem \ref{thm-id}]
Since the proofs of these identities are similar, here we only illustrate the proof using two examples.

For the identity \eqref{G-6-3}, we first square both sides and consider the identity
\begin{align}\label{square-id}
\Pi_{q^2}\Pi_{q^6}\left(\Pi_{q}^2-3\Pi_{q^3}^2 \right)^2=\Pi_q\Pi_{q^3}\left(\Pi_{q^2}^2+3\Pi_{q^6}^2 \right)^2.
\end{align}
By Proposition \ref{prop},  both sides belong to $M_6(\Gamma_0(12))$. According to Lemma \ref{lem-Sturm}, by verifying that the first 13 coefficients of both sides match, we deduce that \eqref{square-id} holds. Taking square roots on both sides and comparing the leading coefficients, we see that \eqref{G-6-3} is true by Lemma \ref{lem-square}.

Our second example is the identity \eqref{G-10-4}. Replacing $q$ by $q^2$ in \eqref{G-10-4}, we see that the product on each side gives a modular form in the space $M_{10}(\Gamma_0(40))$. By Lemma \ref{lem-Sturm}, after verifying that the first 61 coefficients of both sides agree with each other, we prove this identity.
\end{proof}

\begin{rem}\label{rem-square}
As mentioned in the introduction, El Bachraoui \cite{Bachraoui} proved that the square of both sides of \eqref{G-6-2} are equal. Now by Lemma \ref{lem-square} we see that \eqref{G-6-2} holds.
\end{rem}

\section{Proofs of Gosper's Lambert series identities}\label{sec-Lambert}

In this section, we prove all the Lambert series identities stated by Gosper \cite{Gosper}. The method of our proof is quite standard and has been used by He \cite{HeAAM} in his proof of \eqref{LL9-2}. Comparing with He's verification of \eqref{LL9-2}, our proofs here contain more information as we can reproduce Gosper's identities without knowing them in advance.

Let us begin with some preparations. Let
\begin{align}
E_2(z)&:=1-24\sum_{n=1}^\infty \frac{nq^n}{1-q^n}=1-24\sum_{n=1}^\infty \sigma(n)q^n, \label{E2} \\
E_4(z)&:=1+240\sum_{n=1}^\infty \frac{n^3q^n}{1-q^n}=1+240\sum_{n=1}^\infty \sigma_3(n)q^n \label{E4}
\end{align}
be the classical Eisenstein series of weights 2 and 4, respectively. Here $\sigma_s(n):=\sum_{d|n,d>0}d^s$ and $\sigma(n)=\sigma_1(n)$.

Differentiating both sides of
\begin{align}\label{x-frac}
\frac{1}{1-x}=\sum_{n=0}^\infty x^n,
\end{align}
we obtain
\begin{align}
\frac{x}{(1-x)^2}=\sum_{n=1}^\infty nx^n.
\end{align}
This implies
\begin{align}\label{E2-exp}
\sum_{m\geq 1}\frac{q^m}{(1-q^m)^2}=\sum_{m,k\geq 1}mkq^{mk}=\sum_{n\geq 1} \sigma(n)q^n=\frac{1}{24}\left(1-E_2(z)\right).
\end{align}
We will also need the fact that \cite[Exercise 1.2.8(e)]{Diamond-Shurman}
\begin{align}\label{E2-modular}
E_2(z)-NE_2(Nz)\in M_2(\Gamma_0(N)).
\end{align}
\begin{proof}[Proof of Theorem \ref{thm-Lambert}]
We proceed our proofs according to different levels.

(1) (\textbf{Level 2}) Differentiating both sides of \eqref{x-frac} three times, we obtain
\begin{align}\label{x-frac-4}
\frac{1}{(1-x)^4}=\sum_{n=0}^\infty \frac{1}{6}(n+3)(n+2)(n+1)x^n.
\end{align}
Therefore,
\begin{align}
\frac{x(1+4x+x^2)}{(1-x)^4}=\sum_{n=1}^\infty n^3x^n.
\end{align}
Hence,
\begin{align*}
&6\sum_{n\geq 1}\frac{q^{4n-2}}{(1-q^{2n-1})^4}+\sum_{n\geq 1}\frac{q^{2n-1}}{(1-q^{2n-1})^2}=\sum_{n\geq 1}\frac{q^{2n-1}(1+4q^{2n-1}+q^{4n-2})}{(1-q^{2n-1})^4} \\
=&\sum_{n\geq 1}\sum_{k=1}^\infty k^3q^{(2n-1)k} =\sum_{k=1}^\infty k^3\sum_{n=1}^\infty q^{2nk-k} =\sum_{k=1}^\infty \frac{k^3q^k}{1-q^{2k}}.
\end{align*}
This proves the second equality in \eqref{LL1-1}.

To prove the first equality, we observe that
\begin{align}
\sum_{k=1}^\infty \frac{k^3q^k}{1-q^{2k}}=\sum_{k=1}^\infty \sum_{m=0}^\infty k^3q^{k(2m+1)} =\sum_{n=1}^\infty \Big(\sum_{k|n, ~ \frac{n}{k}\equiv 1 \!\! \pmod{2}} k^3 \Big)q^n=:\sum_{n=1}^\infty a(n)q^n.
\end{align}
Let $n=2^sn_0$ with $s\geq 0$ and $n_0$ odd. We have
\begin{align}\label{a-exp}
a(n)=\sum_{d|n_0}(2^sd)^3=2^{3s}\sigma_3(n_0)=\left(\sigma_3(2^s)-\sigma_3(2^{s-1})\right)\sigma_3(n_0)=\sigma_3(n)-\sigma_3(n/2).
\end{align}
Here we agree that $\sigma_3(x)=0$ when $x$ is not an integer.

From \eqref{a-exp} we deduce that
\begin{align}
\sum_{k=1}^\infty \frac{k^3q^k}{1-q^{2k}}=\sum_{n=1}^\infty \left( \sigma_3(n)-\sigma_3(n/2)\right)q^n=\frac{1}{240}\left(E_4(z)-E_4(2z)\right).
\end{align}
Since $E_4(z)$ is a modular form on $\mathrm{SL}_2(\mathbb{Z})$ of weight 4, we conclude that
$$E_4(z)-E_4(2z)\in M_4(\Gamma_0(2)).$$
Note that $\dim M_4(\Gamma_0(2))=2$. It is easy to see that $E_4(z)$ and $E_4(2z)$ form a basis for $M_4(\Gamma_0(2))$. Since $\Pi_q^4\in M_4(\Gamma_0(2))$,  by checking the coefficients of the first two terms in its $q$-expansion, we immediately obtain
\begin{align}
\Pi_q^4=\frac{1}{240}\left(E_4(z)-E_4(2z)\right).
\end{align}
This completes the proof of \eqref{LL1-1}.

(2) (\textbf{Level 4})  We first note that
\begin{align}
1+24\sum_{n\geq 1}\frac{q^n}{(1-q^n)^2}-48\sum_{n\geq 1}\frac{q^{2n}}{(1-q^{2n})^2}=2E_2(2z)-E_2(z)
\end{align}
belongs to $M_2(\Gamma_0(2))$. By Proposition \ref{prop} we see that both
\begin{align*}
\Pi_{q^2}^2 \quad \text{and} \quad \frac{\Pi_q^4}{\Pi_{q^2}^2}
\end{align*}
belong to $M_2(\Gamma_0(4))$. It is easy to see that these two functions are linearly independent. Since $\dim M_2(\Gamma_0(4))=2$, they form a basis of $M_2(\Gamma_0(4))$. Hence there exist constants $c_1,c_2$ such that
\begin{align*}
2E_2(2z)-E_2(z)=c_1\Pi_{q^2}^2+c_2\frac{\Pi_q^4}{\Pi_{q^2}^2}.
\end{align*}
Comparing the first two terms in their $q$-expansions, we deduce that $c_1=16$ and $c_2=1$. This proves \eqref{LL2-1}.

Similarly, using the fact that both
\begin{align*}
4E_2(4z)-E_2(z)=3+24\Big(\sum_{n\geq 1}\frac{q^n}{(1-q^n)^2}-4\sum_{n\geq 1}\frac{q^{4n}}{(1-q^{4n})^2}\Big)
\end{align*}
and
\begin{align*}
&24\Big(\sum_{n\geq 1}\frac{q^{2n-1}}{(1-q^{2n-1})^2}-2\sum_{n\geq 1}\frac{q^{4n-2}}{(1-q^{4n-2})^2}\Big) \\
&=(E_2(2z)-E_2(z))-2(E_2(4z)-E_2(2z)) \\
&=\left(2E_2(2z)-E_2(z) \right)-\left(2E_2(4z)-E_2(2z)\right)
\end{align*}
belong to $M_2(\Gamma_0(4))$,  by checking the first two terms in their series expansions, we get \eqref{LL2-2} and the first equality in \eqref{LL2-3}.

Note that
\begin{align*}
&\sum_{n=1}^\infty \frac{(2n-1)q^{2n-1}}{1-q^{4n-2}}=\sum_{n=1}^\infty \sum_{m=0}^\infty (2n-1)q^{(2n-1)(2m+1)} \nonumber \\
&=\sum_{n=0}^\infty \sigma(2n+1)q^{2n+1} =\sum_{n=1}^\infty \sigma(n)q^n-\sum_{n=1}^\infty \sigma(2n)q^{2n} \\
&=\sum_{n=1}^\infty \sigma(n)q^n-3\sum_{n=1}^\infty \sigma(n)q^{2n}+2\sum_{n=1}^\infty \sigma(n)q^{4n} \\
&=\frac{1}{24}\left(3E_2(2z)-E_2(z)-2E_2(4z)\right)
\end{align*}
belongs to $M_2(\Gamma_0(4))$. The second equality in \eqref{LL2-3} follows by checking the first two terms in the series expansions of both sides.

(3) (\textbf{Level 6})  We first note that
\begin{align}
2+24\sum_{n\geq 1}\frac{q^n}{(1-q^n)^2}-72\sum_{n\geq 1}\frac{q^{3n}}{(1-q^{3n})^2}=3E_2(3z)-E_2(z)
\end{align}
belongs to $M_2(\Gamma_0(3))$.  By Proposition \ref{prop} it is easy to verify that all the functions
\begin{align*}
\frac{\Pi_q^3}{\Pi_{q^3}}, \quad \Pi_q\Pi_{q^3}, \quad \frac{\Pi_{q^3}^3}{\Pi_q}
\end{align*}
belong to $M_2(\Gamma_0(6))$ and are linearly independent. Since $\dim M_2(\Gamma_0(6))=3$, they form a basis of $M_2(\Gamma_0(6))$. By comparing the first three terms, we see that \eqref{LL3-1} holds.

Next, note that
\begin{align*}
&24\Big(\sum_{n\geq 1}\frac{q^{2n-1}}{(1-q^{2n-1})^2}-3\sum_{n\geq 1}\frac{q^{6n-3}}{(1-q^{6n-3})^2}\Big) \\
=&\left(E_2(2z)-E_2(z)\right)-3\left(E_2(6z)-E_2(3z)\right) \\
=&\left(3E_2(3z)-E_2(z)\right)-\left(3E_2(6z)-E_2(2z)\right)
\end{align*}
belongs to $M_2(\Gamma_0(6))$. Using the aforementioned basis, we get \eqref{LL3-2} immediately.

(4) (\textbf{Level 8}) We have $\dim M_2(\Gamma_0(8))=3$. By Proposition \ref{prop} we see that both
\begin{align*}
\Pi_{q^2}^2 \quad \text{and} \quad \frac{\Pi_q^4}{\Pi_{q^2}^2}
\end{align*}
belong to $M_2(\Gamma_0(4))$. Therefore, replacing $q$ by $q^2$, we know that both
\begin{align*}
\Pi_{q^4}^2 \quad \text{and} \quad \frac{\Pi_{q^2}^4}{\Pi_{q^4}^2}
\end{align*}
belong to $M_2(\Gamma_0(8))$. It is easy to see any three of $\Pi_{q^2}^2$, $\Pi_{q^4}^2$, $\frac{\Pi_q^4}{\Pi_{q^2}^2}$ and $\frac{\Pi_{q^2}^4}{\Pi_{q^4}^2}$ form a basis for $M_2(\Gamma_0(8))$. Therefore, \eqref{LL4-1} follows by comparing the first three terms in the $q$-expansions of both sides.

We remak here that by checking the first three coefficients, we obtain
\begin{align}
8\Pi_{q^2}^2+16\Pi_{q^4}^2+\frac{\Pi_{q^2}^4}{\Pi_{q^4}^2}=\frac{\Pi_q^4}{\Pi_{q^2}^2},
\end{align}
which is essentially \eqref{G-4-1}.

(5) (\textbf{Level 10})  Note that
\begin{align*}
4+24\Big(\sum_{n\geq 1}\frac{q^n}{(1-q^n)^2}-5\sum_{n\geq 1}\frac{q^{5n}}{(1-q^{5n})^2} \Big)=5E_2(5z)-E_2(z) \in M_2(\Gamma_0(5)),
\end{align*}
and
\begin{align}
&24\Big(\sum_{n\geq 1}\frac{q^{2n-1}}{(1-q^{2n-1})^2}-5\sum_{n\geq 1}\frac{q^{10n-5}}{(1-q^{10n-5})^2}\Big) \nonumber \\
=&\left(E_2(2z)-E_2(z) \right)-5\left(E_2(10z)-E_2(5z) \right) \nonumber  \\
=&\left(5E_2(5z)-E_2(z) \right)-\left(5E_2(10z)-E_2(2z) \right) \in M_2(\Gamma_0(10)). \label{level10-proof}
\end{align}
Furthermore, by Proposition \ref{prop} it is easy to see that all the functions
\begin{align*}
\Pi_q^3\Pi_{q^5}, \quad \Pi_q^2\Pi_{q^5}^2, \quad \Pi_q\Pi_{q^5}^3
\end{align*}
belong to $M_4(\Gamma_0(10))$. After multiplying both sides of \eqref{LL5-1} by $\Pi_{q}^2\Pi_{q^5}^2$, we see that both sides belong to $M_6(\Gamma_0(10))$. By Lemma \ref{lem-Sturm}, after verifying that the first 10 coefficients of both sides match, we prove \eqref{LL5-1}.

In the same way, by Lemma \ref{lem-Sturm}, by verifying the first 7 terms, we prove that
\begin{align}\label{LL5-proof}
\Big(\sum_{n\geq 1}\frac{q^{2n-1}}{(1-q^{2n-1})^2}-5\sum_{n\geq 1}\frac{q^{10n-5}}{(1-q^{10n-5})^2}\Big)^2=\Pi_q^3\Pi_{q^5}-2\Pi_{q}^2\Pi_{q^5}^2+5\Pi_q\Pi_{q^5}^3
\end{align}
Dividing both sides of \eqref{LL5-proof} by $\Pi_{q^5}^4$ and taking square roots, by Lemma \ref{lem-square} we obtain \eqref{LL5-2}.

(6) (\textbf{Level 12})  It is easy to see that all the functions
\begin{align*}
\Pi_{q^2}^2, \quad \Pi_{q^2}\Pi_{q^6}, \quad \Pi_q\Pi_{q^3}, \quad \Pi_{q^6}^2
\end{align*}
belong to $M_2(\Gamma_0(12))$. Furthermore, note that
\begin{align*}
&24\Big(\sum_{n\geq 1}\frac{q^{2n-1}}{(1-q^{2n-1})^2}-6\sum_{n\geq 1}\frac{q^{12n-6}}{(1-q^{12n-6})^2} \Big) \\
=&\left(E_2(2z)-E_2(z)\right)-6\left(E_2(12z)-E_2(6z)\right) \\
=&\left(6E_2(6z)-E_2(z)\right)-\left(6E_2(12z)-E_2(2z) \right)
\end{align*}
belongs to $M_2(\Gamma_0(12))$. By Lemma \ref{lem-Sturm}, checking that the first 5 terms match in the $q$-expansions of both sides, we get the two equalities in \eqref{LL6-1}.

(7) (\textbf{Level 18})  By Proposition \ref{prop} we see that both $\frac{\Pi_{q^3}^3}{\Pi_q}$ and $\frac{\Pi_{q^3}^3}{\Pi_{q^9}}$
belong to $M_2(\Gamma_0(18))$ and hence is in $M_2(\Gamma_0(18))$.

Note that
\begin{align}
8+24\Big(\sum_{n\geq 1}\frac{q^{2n}}{(1-q^{2n})^2}-\sum_{n\geq 1}\frac{q^{18n}}{(1-q^{18n})^2} \Big)=9E_2(18z)-E_2(2z)
\end{align}
belongs to $M_2(\Gamma_0(18))$. By Lemma \ref{lem-Sturm}, verifying that the first 7 terms on both sides of \eqref{LL9-1} match, we see that \eqref{LL9-1} holds.

Next, using Proposition \ref{prop}, we can check that all the functions
\begin{align}
&\Pi_q^2\Pi_{q^9}^2\left(\frac{\Pi_q}{\Pi_{q^9}} \right)^{\frac{3}{2}}=\Pi_{q}^{\frac{7}{2}}\Pi_{q^{9}}^{\frac{1}{2}}, \quad
\Pi_q^2\Pi_{q^9}^2\frac{\Pi_q}{\Pi_{q^9}} =\Pi_q^3\Pi_{q^9}, \quad \Pi_q^2\Pi_{q^9}^2\left(\frac{\Pi_q}{\Pi_{q^9}} \right)^{\frac{1}{2}}=\Pi_{q}^{\frac{5}{2}}\Pi_{q^{9}}^{\frac{3}{2}},  \notag \\
&\Pi_q\Pi_{q^9}^3\left(\frac{\Pi_q}{\Pi_{q^9}} \right)^{\frac{3}{2}}=\Pi_q^{\frac{5}{2}}\Pi_{q^9}^{\frac{3}{2}},  \quad \Pi_q\Pi_{q^9}^3\frac{\Pi_q}{\Pi_{q^9}}=\Pi_q^2\Pi_{q^9}^2, \quad
\Pi_q\Pi_{q^9}^3\left(\frac{\Pi_q}{\Pi_{q^9}} \right)^{\frac{1}{2}}=\Pi_q^{\frac{3}{2}}\Pi_{q^9}^{\frac{5}{2}},   \label{LL9-proof-functions} \\
&\Pi_{q^9}^4\left(\frac{\Pi_q}{\Pi_{q^9}} \right)^{\frac{3}{2}}=\Pi_q^{\frac{3}{2}}\Pi_{q^9}^{\frac{5}{2}}, \quad
\Pi_{q^9}^4\frac{\Pi_q}{\Pi_{q^9}}=\Pi_q\Pi_{q^9}^3,   \quad \Pi_{q^9}^4\left(\frac{\Pi_q}{\Pi_{q^9}} \right)^{\frac{1}{2}}=\Pi_q^{\frac{1}{2}}\Pi_{q^9}^{\frac{7}{2}}  \notag
\end{align}
belong to $M_4(\Gamma_0(18))$.

Since
\begin{align*}
&24\Big(\sum_{n\geq 1}\frac{q^{2n-1}}{(1-q^{2n-1})^2}-9\sum_{n\geq 1}\frac{q^{18n-9}}{(1-q^{18n-9})^2}\Big) \\
=&\left(E_2(2z)-E_2(z)\right)-9\left(E_2(18z)-E_2(9z) \right) \\
=&\left(9E_2(9z)-E_2(z) \right)-\left(9E_2(18z)-E_2(2z) \right)
\end{align*}
belongs to $M_2(\Gamma_0(18))$. Taking square on both sides of \eqref{LL9-2} and expanding the right side, we see that the right side is a linear combination of the functions listed in \eqref{LL9-proof-functions}. Hence the squares of both sides of \eqref{LL9-2} belong to $M_4(\Gamma_0(18))$.  Therefore, by Lemma \ref{lem-Sturm}, after checking that the first 13 terms of both sides match, we conclude that the squares of both sides of \eqref{LL9-2} are equal. By Lemma \ref{lem-square} we see that \eqref{LL9-2} holds.

Next, note that
\begin{align}\label{Level18-proof-eq}
24\Big(\sum_{n\geq 1}\frac{q^n}{(1-q^n)^2}-9\sum_{n\geq 1}\frac{q^{9n}}{(1-q^{9n})^2} \Big)+8 =9E_2(9z)-E_2(z)\in M_2(\Gamma_0(9)).
\end{align}
By Proposition \ref{prop} we see that
\begin{align*}
\Pi_{q}^4\in M_4(\Gamma_0(2)), \quad \Pi_{q}^3\Pi_{q^{9}}\in M_4(\Gamma_0(18)), \quad  \Pi_q^2\Pi_{q^9}^2 \in M_4(\Gamma_0(18)).
\end{align*}
After taking square on both sides of \eqref{LL9-3} and multiplying both sides by $\Pi_q^3\Pi_{q^9}$, we see that the resulting functions on both sides belong to $M_8(\Gamma_0(18))$. By Lemma \ref{lem-Sturm}, after checking that the first 25 terms on both sides match, we deduce that the squares of both sides of \eqref{LL9-3} are equal. By Lemma \ref{lem-square} we see that \eqref{LL9-3} holds.

Note that all the functions
\begin{align*}
\Pi_q\Pi_{q^3}, \quad \Pi_{q^3}\Pi_{q^9}, \quad \frac{\Pi_q^3}{\Pi_{q^3}}, \quad \frac{\Pi_{q^3}^3}{\Pi_{q^9}}, \quad \frac{\Pi_{q^3}^3}{\Pi_q}, \quad \frac{\Pi_{q^9}^3}{\Pi_{q^3}}
\end{align*}
belong to $M_2(\Gamma_0(18))$. By Lemma \ref{lem-Sturm}, by checking the first 7 terms in their series expansions, we obtain
\begin{align}
6\Pi_q\Pi_{q^3}+18\Pi_{q^3}\Pi_{q^9}-\frac{\Pi_q^3}{\Pi_{q^3}}+ \frac{\Pi_{q^3}^3}{\Pi_{q^9}}+3\frac{\Pi_{q^3}^3}{\Pi_q}-27\frac{\Pi_{q^9}^3}{\Pi_{q^3}}=0.
\end{align}
Dividing both sides by $\Pi_{q^3}^2$, we obtain \eqref{L18-5}. Next, by \eqref{Level18-proof-eq} and verifying the first 7 terms, we obtain \eqref{La18-4}.

(8) (\textbf{Level 20}) Recall \eqref{level10-proof}. By Proposition \ref{prop} we see that the functions
\begin{align*}
\Pi_q^3\Pi_{q^5}, \quad \Pi_q^2\Pi_{q^5}^2, \quad \Pi_{q^5}\Pi_q^3
\end{align*}
all belong to $M_4(\Gamma_0(10))$, and the functions
\begin{align*}
\Pi_q^2\Pi_{q^5}^6, \quad \Pi_q^2\Pi_{q^5}^2\Pi_{q^{10}}^4
\end{align*}
all belong to $M_8(\Gamma_0(20))$. Therefore, by Lemma \ref{lem-Sturm}, after verifying that the first 25 terms agree, we deduce that
\begin{align*}
&\Big( \sum_{n\geq 1}\frac{q^{2n-1}}{(1-q^{2n-1})^2}-5 \sum_{n\geq 1}\frac{q^{10n-5}}{(1-q^{10n-5})^2}\Big)\left(\Pi_q^3\Pi_{q^5}-4\Pi_q^2\Pi_{q^5}^2-\Pi_q\Pi_{q^5}^3 \right)\Pi_{q^{10}}^2 \nonumber \\
&=\Pi_q^2\Pi_{q^5}^6+16\Pi_q^2\Pi_{q^5}^2\Pi_{q^{10}}^4
\end{align*}
since both sides belong to $M_8(\Gamma_0(20))$. After dividing both sides by $\Pi_{q}^2\Pi_{q^5}^4\Pi_{q^{10}}^2$ and rearrangements, we obtain \eqref{LL20-1}.
\end{proof}

From the proofs given above, we summarize the general strategy for constructing Gosper's Lambert series identities as follows. First, we use the Eisenstein series to generate a Lambert series which belongs to certain space $M_k(\Gamma_0(N))$. Then we use $\Pi_{q}$ to generate a basis or part of a basis of the same space. Finally, we use the first few terms (according to Lemma \ref{lem-Sturm}) in their series expansions to check whether there are linear relations between the Lambert series and the basis elements generated by $\Pi_q$. If so, this will give a Lambert series identity involving $\Pi_q$.

%We end this section by giving one new identity discovered by the above strategy:
%\begin{itemize}
%\item Level 8
%\begin{align}
%3\sum_{n=1}^\infty \frac{q^{2n-1}}{(1-q^{2n-1})^2}-7\sum_{n=1}^\infty \frac{q^{4n-2}}{(1-q^{4n-2})^2}+4\sum_{n=1}^\infty \frac{q^{8n-4}}{(1-q^{8n-4})^2}=3\Pi_{q^2}^2+2\Pi_{q^4}^2. \tag{La8-2} \label{LL8-2}
%\end{align}
%\end{itemize}
%The identity \eqref{LL8-2} is essentially
%\begin{align}
%3E_2(z)-7E_2(2z)+4E_2(8z)+72\Pi_{q^2}^2+48\Pi_{q^4}^2=0.
%\end{align}

\section{Construction of Gosper's identities using hauptmoduls}\label{sec-hauptmodul}

We may also use another way based on weight 0 modular functions to construct and prove some of Gosper's identities. This method works very well for those identities involving fractional expressions of $\Pi_q$ such as \eqref{G-4-1}.

Recall that for a genus zero congruence subgroup $\Gamma$ of $\mathrm{SL}_2(\mathbb{R})$ commensurable with $\mathbb{SL}_2(\mathbb{Z})$, the function field  of $\Gamma\backslash \mathbb{H}^*$ over $\mathbb{C}$ can be generated by a single modular function, and such function is called a hauptmodul  if it has a unique simple pole of residue 1 at the cusp $\infty$.

The following lemma gives a way to construct a hauptmodul using eta products.
\begin{lemma}\label{lem-hauptmodul}
(Cf.\ \cite{Kondo}) Let $\Gamma$ be a discrete subgroup of $\mathrm{SL}_2(\mathbb{Z})$ containing $\Gamma_0(N)$ for some $N$ such that $\Gamma \backslash \mathbb{H}^*$ is of genus 0. Let $h(z)=\prod\limits_{\delta|N}\eta(\delta z)^{r_\delta}$ be an eta product with $r_\delta\in \mathbb{Z}$. Suppose
\begin{enumerate}[$(1)$]
\item $\sum_{\delta|N} \delta r_\delta=-24$;
\item $h(z)$ is invariant under the action of $\Gamma$;
\item $\Gamma_\infty=\{\alpha\in \Gamma: \alpha(\infty)=\infty\}=\left\{\begin{pmatrix}
\pm 1 & n \\ 0 & \pm 1
\end{pmatrix}: n\in \mathbb{Z} \right\}$, and
\item $z=\infty$ is the unique pole of $h(z)$ among all inequivalent cusps of $\Gamma$.
\end{enumerate}
Then $h(z)$ generates the function field of $\Gamma\backslash \mathbb{H}^*$ over $\mathbb{C}$, i.e., $h(z)$ is a bijection from $\Gamma\backslash \mathbb{H}^*\rightarrow \mathbb{C}\cup \{\infty\}$ and every meromorphic function invariant under $\Gamma$ can be expressed as a rational function of $h(z)$.
\end{lemma}

We will restrict our attention on $\Gamma=\Gamma_0(N)$. Let $M_k^{!}(\Gamma_0(N))$ be the space of weakly holomorphic modular forms on $\Gamma_0(N)$ of even integer weight $k$, i.e., meromorphic modular forms of weight $k$ on $\Gamma_0(N)$ which can only have poles at cusps. We will focus on $M_0^{!}(\Gamma_0(N))$, which is a subset of the function field  of $\Gamma_0(N)\backslash \mathbb{H}^*$.

It is known \cite[Theorem 15, p.\ 103]{Schoen} that $\Gamma_0(N)$ has genus zero for $N\leq 10$ and $N\in \{12,13,16,18,25\}$. For each genus zero group $\Gamma_0(N)$, we know hauptmoduls in terms of eta products exist. Let $h(z)$ be a hauptmodul for $\Gamma_0(N)$. Given a function $F(z)$ in $M_0^{!}(\Gamma_0(N))$, we know that $F(z)$ can be expressed as a rational function of $h(z)$. This will produce an identity. In particular, if we are able to use $\Pi_q$ to construct the functions $h(z)$ and $F(z)$, then we will get a $\Pi_q$-identity.

Now we will use this strategy to reprove some of Gosper's identities in genus zero levels 8, 12, 16 and 18, which will suffice for illustrating this method. Note that the hauptmodul $h(z)$ and auxiliary function $F(z)$ are not the same for different levels.

\subsection{Level 8}
By Proposition \ref{prop} and Lemma \ref{lem-hauptmodul}, we see that the function
\begin{align*}
h(z)=\frac{\Pi_{q^2}^2}{\Pi_{q^4}^2}=\frac{\eta^{12}(4z)}{\eta^4(2z)\eta^8(8z)}=q^{-1}+4q+2q^3+O(q^4)
\end{align*}
is a hauptmodul for $\Gamma_0(8)$. It is not difficult to find the values of $h(z)$ at cusps other than $\infty$ of $\Gamma_0(8)$:
\begin{align}
h(0)=4, \quad h(\frac{1}{2})=-4, \quad h(\frac{1}{4})=0.
\end{align}

Consider
\begin{align*}
F_1(z):=\frac{\Pi_q^2}{\Pi_{q^2}\Pi_{q^4}},
%F_2(z):=\frac{\Pi_q^2}{\Pi_{q^2}\Pi_{q^4}}=\frac{\eta^{10}(2z)}{\eta^4(z)\eta^2(4z)\eta^4(8z)}, \\
\end{align*}
which belongs to $M_0^{!}(\Gamma_0(8))$. Comparing the orders of $F_1(z)$ and $h(z)$ at the cusps (see Table \ref{Tab-level8}), we see that $F_1(z)$ and $h(z)-h(\frac{1}{2})$ have the same zeros and poles. Hence there exists some constant $c_1$ such that
$$F_1(z)=c_1\Big(h(z)-h(\frac{1}{2})\Big)=c_1(h(z)+4).$$
Note that $F_1(z)=q^{-1}(1+O(q))$. We deduce that $c_1=1$, and this proves \eqref{G-4-1}.
\begin{table}[htbp]
\renewcommand\arraystretch{1.5}
\begin{tabular}{c|cccc}
  \hline
  % after \\: \hline or \cline{col1-col2} \cline{col3-col4} ...
  cusp $p$ & $\infty$ & $0$ & $\frac{1}{2}$ & $\frac{1}{4}$ \\
  \hline
  $\mathrm{ord}(h,p)$ & $-1$ & 0 & 0 & 1\\
  \hline
   $\mathrm{ord}(F_1,p)$ & $-1$ & 0 & 1 & 0  \\
   \hline
\end{tabular}
\vspace{1mm}
\caption{Orders of functions at cusps of $\Gamma_0(8)$}\label{Tab-level8}
\end{table}

\subsection{Level 12}
By Lemma \ref{lem-hauptmodul} we know that the function
\begin{align*}
h(z):=\frac{\Pi_{q^2}}{\Pi_{q^6}}=\frac{\eta^4(4z)\eta^2(6z)}{\eta^2(2z)\eta^4(12z)}=q^{-1}+2q+q^3+O(q^3)
\end{align*}
is a hauptmodul for $\Gamma_0(12)$. It is not difficult to show that
\begin{align}
h(0)=3, \quad h(\frac{1}{2})=-3, \quad h(\frac{1}{4})=0, \quad h(\frac{1}{3})=-1, \quad h(\frac{1}{6})=1.
\end{align}

Consider the functions
\begin{align*}
&F_1(z):=\frac{\Pi_{q}\Pi_{q^3}}{\Pi_{q^6}^2}, \quad F_2(z):=\frac{\Pi_{q^3}^2}{\Pi_q^2}, \quad F_3(z):=\frac{\Pi_{q^3}^4}{\Pi_{q^6}^4}, \\ &F_4(z):=\frac{\Pi_q^4}{\Pi_{q^6}^4}, \quad F_5(z):=\frac{\Pi_{q^3}^3}{\Pi_q\Pi_{q^6}^2}, \quad F_6(z):=\frac{\Pi_{q}^3}{\Pi_{q^3}\Pi_{q^6}^2}.
\end{align*}

They all belong to $M_0^{!}(\Gamma_0(12))$ and we have
\begin{align}\label{F-relation}
F_3(z)=F_1^2(z)F_2(z), \quad F_4(z)=\frac{F_1^2(z)}{F_2(z)}, \quad F_5(z)=F_1(z)F_2(z), \quad F_6(z)=\frac{F_1(z)}{F_2(z)}.
\end{align}

\begin{table}[htbp]
\renewcommand\arraystretch{1.5}
\begin{tabular}{c|cccccc}
  \hline
  % after \\: \hline or \cline{col1-col2} \cline{col3-col4} ...
  cusp $p$ & $\infty$ & $0$ & $\frac{1}{2}$ & $\frac{1}{3}$ & $\frac{1}{4}$ & $\frac{1}{6}$ \\
  \hline
  $\mathrm{ord}(h,p)$ & $-1$ & 0 & 0 & 0 & 1 & 0 \\
  \hline
   $\mathrm{ord}(F_1,p)$ & $-2$ & 0 & 1 & 0 & 0 & 1 \\
   \hline
  $\mathrm{ord}(F_2,p)$ & 1 & 0 & $-1$ & 0 & $-1$ & 1 \\
  \hline
 % $\mathrm{ord}(F_3,p)$ & $-3$ & 0 & 1 & 0 & $-1$ & 3 \\
%  \hline
%  $\mathrm{ord}(F_4,p)$ & $-5$ & 0 & 3 & 0 & 1 & 1 \\
%  \hline
%   $\mathrm{ord}(F_5,p)$ & $-1$ & 0 & 0 & 0 & $-1$ & 2 \\
%   \hline
%    $\mathrm{ord}(F_6,p)$ & $-3$ & 0 & 2 & 0 & 1 & 0 \\
%  \hline
\end{tabular}
\vspace{1mm}
\caption{Orders of functions at cusps of $\Gamma_0(12)$}\label{tab-order}
\end{table}

Comparing the orders of $F_i(z)$ ($i=1,2$) and $h(z)$ at the cusps (see Table \ref{tab-order}),  we deduce that
\begin{align}\label{F3}
F_1(z)&=\Big(h(z)-h(\frac{1}{2}) \Big)\Big(h(z)-h(\frac{1}{6}) \Big)=\left(h(z)-1 \right)\left(h(z)+3 \right), \\
F_2(z)&=\frac{h(z)-h(\frac{1}{6})}{(h(z)-h(\frac{1}{2}))(h(z)-h(\frac{1}{4}))}=\frac{h(z)-1}{h(z)(h(z)+3)}.
\end{align}
This proves \eqref{G-6-1} and \eqref{G-6-2}. Now using \eqref{F-relation} we obtain
\begin{align*}
F_3(z)&=\frac{(h(z)+3)(h(z)-1)^3}{h(z)}, \quad
F_4(z)=h(z)(h(z)+3)^3(h(z)-1),\\
F_5(z)&=\frac{(h(z)-1)^2}{h(z)}, \quad F_6(z)=h(z)(h(z)+3)^2.
\end{align*}
This proves  \eqref{G-6-4}, \eqref{G-6-5}, \eqref{L6-13} and \eqref{L6-14}, respectively.

\subsection{Level 16}

By Lemma \ref{lem-hauptmodul} we know that the function
\begin{align*}
h(z):=\frac{\Pi_{q^4}}{\Pi_{q^8}}=\frac{\eta^6(8z)}{\eta^2(4z)\eta^4(16z)}
\end{align*}
is a hauptmodul for $\Gamma_0(16)$. We find the values of $h$ at cusps other than $\infty$ of $\Gamma_0(16)$:
\begin{align*}
h(0)=2, \quad h(\frac{1}{2})=-2, \quad h(\frac{1}{4})=-2i, \quad h(\frac{3}{4})=2i, \quad h(\frac{1}{8})=0.
\end{align*}

Consider the functions
\begin{align}
F_1(z):=\frac{\Pi_q^2}{\Pi_{q^2}\Pi_{q^8}}, \quad F_2(z):=\frac{\Pi_q^4}{\Pi_{q^2}^4}, \quad F_3(z):=\frac{\Pi_{q^2}^2}{\Pi_{q^4}^2},
\end{align}
which all belong to $M_0^{!}(\Gamma_0(16))$.

Comparing the orders of $F_1(z), F_2(z)$, $F_3(z)$ and $h(z)$ at the cusps (see Table \ref{tab-L16-order}), we conclude that
\begin{align}
F_1(z)&=\Big(h(z)-h(\frac{1}{2})\Big)^2=\Big(h(z)+2\Big)^2, \\
F_2(z)&=\frac{\Big(h(z)-h(\frac{1}{2})\Big)^4}{\Big(h(z)-h(\frac{1}{4})\Big)\Big( h(z)-h(\frac{3}{4}) \Big)\Big(h(z)-h(\frac{1}{8}) \Big)}=\frac{(h(z)+2)^4}{h(z)(h(z)^2+4)}, \\
F_3(z)&=\frac{\Big(h(z)-h(\frac{1}{4})\Big)\Big(h(z)-h(\frac{3}{4})\Big)}{h(z)-h(\frac{1}{8})}=\frac{h(z)^2+4}{h(z)}.
\end{align}
This proves \eqref{L8-1}, \eqref{L8-2} and \eqref{G-4-1} (with $q$ replaced by $q^2$), respectively.

\begin{table}[htbp]
\renewcommand\arraystretch{1.5}
\begin{tabular}{c|cccccc}
  \hline
  % after \\: \hline or \cline{col1-col2} \cline{col3-col4} ...
  cusp $p$ & $\infty$ & $0$ & $\frac{1}{2}$ & $\frac{1}{4}$ & $\frac{3}{4}$ & $\frac{1}{8}$  \\
  \hline
  $\mathrm{ord}(h,p)$ & $-1$ & 0 & 0 & 0 & 0 & 1 \\
  \hline
   $\mathrm{ord}(F_1,p)$ & $-2$ & 0 & 2 & 0 & 0 & 0 \\
   \hline
  $\mathrm{ord}(F_2,p)$ & $-1$ & 0 & $4$ & $-1$ & $-1$ & $-1$ \\
  \hline
  $\mathrm{ord}(F_3,p)$ & $-1$ & 0 & 0 & 1 & 1 & $-1$ \\
%  \hline
%  $\mathrm{ord}(F_4,p)$ & $-5$ & 0 & 3 & 0 & 1 & 1 \\
%  \hline
%   $\mathrm{ord}(F_5,p)$ & $-1$ & 0 & 0 & 0 & $-1$ & 2 \\
%   \hline
%    $\mathrm{ord}(F_6,p)$ & $-3$ & 0 & 2 & 0 & 1 & 0 \\
  \hline
\end{tabular}
\vspace{1mm}
\caption{Orders of functions at cusps of $\Gamma_0(16)$}\label{tab-L16-order}
\end{table}

\subsection{Level 18}
By Lemma \ref{lem-hauptmodul} we know that
\begin{align}
h(z):=\sqrt{\frac{\Pi_{q}}{\Pi_{q^9}}}=\frac{\eta^2(2z)\eta(9z)}{\eta(z)\eta^2(18z)}=q^{-1}+1+q^2+O(q^3)
\end{align}
is a hauptmodul for $\Gamma_0(18)$. We find that
\begin{align}
h(0)=3, \quad h(\frac{1}{2})=0, \quad h(\frac{1}{3})=-\sqrt{3}i, \quad h(\frac{2}{3})=\sqrt{3}i, \nonumber \\ h(\frac{1}{6})=\frac{3}{2}-\frac{\sqrt{3}}{2}i, \quad h(\frac{5}{6})=\frac{3}{2}+\frac{\sqrt{3}}{2}i, \quad h(\frac{1}{9})=1.
\end{align}

Let
\begin{align}
F_1(z):=\frac{\Pi_{q^3}^2}{\Pi_{q^9}^2},
\end{align}
which belongs to $M_0^{!}(\Gamma_0(18))$. Comparing the orders of $F_1(z)$ and $h(z)$ at cusps (see Table \ref{tab-order-L18}), we conclude that
\begin{align}
F_1(z)=h(z)\Big(h(z)-h(\frac{1}{6})\Big)\Big(h(z)-h(\frac{5}{6})\Big)=h(z)\left(h^2(z)-3h(z)+3\right).
\end{align}
This proves \eqref{G-9-1} and \eqref{L9-2} simultaneously.
\begin{table}[htbp]
\renewcommand\arraystretch{1.5}
\begin{tabular}{c|cccccccc}
  \hline
  % after \\: \hline or \cline{col1-col2} \cline{col3-col4} ...
  cusp $p$ & $\infty$ & $0$ & $\frac{1}{2}$ & $\frac{1}{3}$ & $\frac{2}{3}$ & $\frac{1}{6}$ & $\frac{5}{6}$ & $\frac{1}{9}$  \\
  \hline
  $\mathrm{ord}(h,p)$ & $-1$ & 0 & 1 & 0 & 0 & 0 & 0 & 0 \\
  \hline
   $\mathrm{ord}(F_1,p)$ & $-3$ & 0 & 1 & 0 & 0 & 1 &1  & 0  \\
  \hline
\end{tabular}
\vspace{2mm}
\caption{Orders of functions at cusps of $\Gamma_0(18)$}\label{tab-order-L18}
\end{table}

We take this chance to point out the following fact. Though there are numerous identities for the same index set $\{n_1,\cdots,n_m\}$, some of them are in fact equivalent. For example, the five level 18 identities for $\Pi_q,\Pi_{q^3},\Pi_{q^9}$ are equivalent to each other, which means that we can prove any other four from one of them. Below we show how \eqref{L9-2}, \eqref{L9-3}, \eqref{L9-4} and \eqref{L18-5} can be obtained from \eqref{G-9-1}.
\begin{proof}[Proof of level 18 identities via \eqref{G-9-1}]
We denote $a=\sqrt{\Pi_q}$ and $b=\sqrt{\Pi_{q^9}}$. From \eqref{G-9-1} we have $\Pi_{q^3}^2=ab(a^2+3b^2-3ab)$. Thus
\begin{align}\label{Pi39-1}
\Pi_{q^3}^2-\Pi_{q^9}^2&=ab(a^2+3b^2-3ab)-b^4=b(a^3-3a^2b+3ab^2-b^3) \nonumber \\
&=b(a-b)^3=\sqrt{\Pi_{q^9}}\left(\sqrt{\Pi_q}-\sqrt{\Pi_{q^3}} \right)^3.
\end{align}
This proves \eqref{L9-2}. Next,
\begin{align}\label{Pi39-2}
\Pi_{q}^2-\Pi_{q^3}^2=a^4-ab(a^2+3b^2-3ab)=a(a-b)(a^2+3b^2).
\end{align}
Now by \eqref{Pi39-1} and \eqref{Pi39-2}, we see that both sides of \eqref{L9-3} are equal to $b^4a^3(a-b)^3(a^2+3b^2)^3$. This proves \eqref{L9-3}.
Substituting \eqref{Pi39-1} and \eqref{Pi39-2} into \eqref{L9-4}, we see that both sides of \eqref{L9-4} are equal to $a^6b^2(a-b)^6(a^2+3b^2)^6$, and hence \eqref{L9-4} holds. Similarly, multiplying both sides of \eqref{L18-5} by $\Pi_q\Pi_{q^3}^3\Pi_{q^9}$ and making use of $\Pi_{q^3}^2=ab(a^2+3b^2-3ab)$, we prove \eqref{L18-5}.
\end{proof}

Finally, we remark here that some of  Gosper's Lambert series identities can also be proved using hauptmoduls. For example,
by Lemma \ref{lem-hauptmodul} we know that the function
\begin{align}
h(z):=\frac{\Pi_{q}}{\Pi_{q^5}}=\frac{\eta^4(2z)\eta^2(5z)}{\eta^2(z)\eta^4(10z)}=q^{-1}+2+q+2q^2+O(q^3)
\end{align}
is a hauptmodul for $\Gamma_0(10)$. Consider
\begin{align}
F_1(z):=\frac{6\left(5E_2(5z)-E_2(z) \right)}{5E_2(5z)-E_2(z)-\left(5E_2(10z)-E_2(2z)\right)}.
\end{align}
Then $F_1(z)$ is clearly a modular function invariant on $\Gamma_0(10)$, and hence can be expressed as a rational function of $h(z)$. By computation we find that
\begin{align}
F_1(z)=\frac{h(z)^2+2h(z)+5}{h(z)},
\end{align}
which is \eqref{LL5-1}. But in this case, the function $F_1(z)$ has zeros in $\mathbb{H}$. And in more general cases, the auxiliary functions constructed using Lambert series like $F_1(z)$ may also have poles in $\mathbb{H}$. Analyzing the orders of  zeros and poles for these auxiliary functions will be more complicated than for eta products. Thus for proving Gosper's Lambert series identities, this method does not show advantage compared with the method in Section \ref{sec-Lambert}. Therefore, we do not pursue it here.

\subsection*{Acknowledgements}
This work was supported by the National Natural Science Foundation of China (11801424) and a start-up research grant of the Wuhan University.

\end{document}